\documentclass[11pt]{article} 
\usepackage{amsmath,amsthm} 
\usepackage{amssymb,latexsym}
\usepackage[T1]{fontenc}
\usepackage{authblk}
\usepackage{xcolor}
\usepackage{lipsum}
\usepackage{hyperref}
\usepackage{lipsum} 
\usepackage{lineno}
\modulolinenumbers[1]
\begin{document}
	\newtheorem{theorem}{Theorem}[section]
	\newtheorem{question}{Question}[section]
	\newtheorem{thm}[theorem]{Theorem}
	\newtheorem{lem}[theorem]{Lemma}
	\newtheorem{eg}[theorem]{Example}
	\newtheorem{prop}[theorem]{Proposition}
	\newtheorem{cor}[theorem]{Corollary}
	\newtheorem{rem}[theorem]{Remark}
	\newtheorem{deff}[theorem]{Definition}
	\numberwithin{equation}{section}
	\title{On $S$-Integral Domains and $S$-Version of Krull Intersection Theorem}
	
	\author[1]{Tushar Singh}
	\author[2]{Gyanendra K. Verma}
	\author[3]{Shiv Datt Kumar }

	\affil[1, 3]{\small Department of Mathematics, Motilal Nehru National Institute of Technology Allahabad, Prayagraj 211004, India \vskip0.01in Emails: $^1$ sjstusharsingh0019@gmail.com, tushar.2021rma11@mnnit.ac.in, $^3$sdt@mnnit.ac.in}
	\vskip0.05in
	\affil[2]{\small Department of Mathematical Sciences, UAE University, Al Ain PO Box 15551, United Arab Emirates \vskip0.01in Email: gkvermaiitdmaths@gmail.com}
	\maketitle
	\hrule
	\begin{abstract}
		\noindent
		 Let $S\subseteq R$ be a multiplicatively closed subset of a ring $R$.  We extend several results on integral domains to their $S$-versions and establish the $S$-version of Krull intersection theorem. We also show that if $R$ is an $S$-field, then the localization of $R$ with respect to $S$ is a $\phi(S)$-field, where $\phi(S)=\left \{\dfrac{s}{1}| \ s\in S\right \}$ is a multiplicatively closed subset of $S^{-1}R$, and prove the converse under the condition of finiteness of $S$. As a consequence, we show that every finite $S$-integral domain is an $S$-field. Also, we provide several examples to illustrate the significance of our findings. 
		
	\end{abstract}
	\textbf{Keywords:}	$S$-integral domain, $S$-field, $S$-cancellation property, $S$-Krull intersection theorem.\\
	\textbf{MSC(2020):} 12E20, 13G05, 16U10, 16U40.
	\hrule
	\section{Introduction}
	Throughout, $R$ denotes  a commutative ring with unity, and $S\subseteq R$ denotes a multiplicatively closed subset. The concept of Noetherian modules significantly simplifies the study of ring and module structures. The concepts of $S$-principal ideal rings and $S$-Noetherian rings, which are extensions of principal ideal rings and Noetherian rings, were first presented by Anderson and Dumitrescu \cite{ad02} in 2002. They established $S$-versions of well-known results for Noetherian rings, along with an $S$-version of the Eakin-Nagata theorem. A ring $R$ is defined as an $S$-Noetherian ring if every ideal of $R$ is $S$-finite. This means that for every ideal $I$ in $R$, there exist $s\in S$ and a finitely generated ideal $J$ of $R$ such that $sI\subseteq J\subseteq I$. In $2020$, Sevim et al. \cite{es20} extended the framework of Artinian rings by introducing the concept of $S$-Artinian rings. A ring $R$ is called  \textit{$S$-Artinian} if for every descending chain of ideals $\mathcal{I}_{1}\supseteq \mathcal{I}_{2}\supseteq\cdots\supseteq \mathcal{I}_{n}\supseteq\cdots$ of $R$, there exist $s \in S$ and $k \in \mathbb{N}$ such that $s \mathcal{I}_{k} \subseteq \mathcal{I}_n$ for all $n \geq k$.
	Several authors generalized numerous significant properties of Noetherian and Artinian rings to $S$-variant such as $S$-prime ideals, $S$-strong Mori domains, $S$-Noetherian properties on amalgamated algebras along an ideal, $S$-reduced modules, $S$-cogenerated rings, and $S$-primary decomposition (see \cite{ah18}, \cite{ah20}, \cite{ah22}, \cite{hk14}, \cite{jw14}, \cite{ap20}, \cite{es20}, \cite{ts23}, and \cite{ts24}). Recently, Ersoy et al. \cite{ba24} introduced $S$-version and $S$-generalizations of idempotent elements, pure ideals, and gave $S$-version of Stone type theorems.

	Integral domains are interesting algebraic structures that possess numerous characteristics analogous to those of the ring of integers. A significant characteristic is the absence of zero divisors, which ensures a consistent and precise arithmetic structure. This makes them essential in number theory for the study of divisibility, prime factorization, and Diophantine equations. In algebraic geometry, integral domains are associated with coordinate rings of irreducible varieties, which allow us to analyze geometric objects. In 2020, Yildiz et al. \cite{ed20} introduced the concept of $S$-integral domain with respect to a multiplicatively closed set $S$, which is one of the most interesting generalizations of the integral domains. A ring $R$ is said to be an $S$-integral domain if there exists an $s \in S$ such that for all $a, b \in R$, whenever $ab = 0$, then either $sa = 0$ or $sb = 0$. This approach allows one to generalize the classical integral domain and study algebraic structures where certain elements of the ring are invertible or have special behaviour due to the presence of $S$. The notion of $S$-integral domains is essential in the analysis of the localization of rings and modules and in algebraic geometry.
	By focusing on particular multiplicative closed subsets, we gain greater flexibility in analyzing the structure of rings in various mathematical settings.

	In the first part of the paper, we extend several results from the integral domain to the $S$-integral domain. We present the idea of the $S$-cancellation property in a ring (see Definition \ref{cancel}) and prove that $S$-cancellation property holds in a ring if and only if the ring is an $S$-integral domain (see Proposition \ref{scancellation}). Additionally, we establish that the localization of a ring $R$ with respect to $S$ is an $S$-integral domain if $R$ is an $S$-integral domain; however, the converse is not generally true (see Example \ref{sint}). Under some assumptions, we prove the converse. In an integral domain, the only idempotent elements are $0$ and $1$; however, an $S$-integral domain can have additional $S$-idempotent elements (see Example \ref{sidem}). Moreover, we give sufficient conditions on an integral domain to be an $S$-Noetherian ring (see Theorem \ref{srad}). 
	
	Fields are important algebraic structures that serve as a domain for exploration in various areas, including polynomial theory, algebraic geometry, coding theory, and cryptography. In the second part of this paper, we examine the concept of $S$-fields, which extend the idea of fields as introduced by E. Yildiz et al. \cite{ed20} in 2020. Precisely, a ring $R$ is called an $S$-field if its zero ideal is an $S$-maximal ideal (see Definition \ref{sfield}).  Extending the concept of modules over rings to modules over $S$-fields is an interesting and important problem. Such a generalization naturally encompasses vector spaces as a special case, thereby enriching and extending the classical framework. In doing so, it provides a foundation for the systematic analysis of algebraic structures within commutative algebra from a new perspective. Moreover, this extension is not merely of theoretical interest; it opens promising directions for applied research, particularly in areas such as coding theory, cryptography, and related areas where the interplay between vector spaces and field theory plays a foundational role.    
    In this work, we extend several fundamental properties of fields to $S$-fields and give a characterization of $S$-fields. We define $S$-proper ideals of a ring (see Definition \ref{sprop}) and prove that a ring is an $S$-field if and only if it has no $S$-proper ideal. We also provide the characterization of $S$-maximal and $S$-prime ideals of a ring in connection with $S$-fields (see Proposition \ref{mfield} and  Theorem \ref{uni}). Additionally, we investigate the localization of a ring with respect to a multiplicatively closed set $S$ and determine the conditions under which the localization of a ring becomes a $\phi(S)$-field. As a corollary, we show that every finite $S$-integral domain is an $S$-field. 
	The Krull intersection theorem is a cornerstone in the theory of commutative Noetherian rings, providing deep insight into the structure and stability of ideals under infinite intersections. Building upon this classical framework, we present an $S$-analogue of this classical result, which we call the $S$-Krull intersection theorem (see Theorem \ref{krull}).

\section{On $S$-integral domains}
In this section, we first provide an example of an $S$-integral domain which is not an integral domain and demonstrate that if $R$ is an $S$-integral domain, then the localization $S^{-1}R$ is an integral domain. However, the converse does not necessarily hold. We characterize the $S$-idempotent and $S$-nilpotent elements in $S$-integral domains. Then we prove that if $R$ is an integral domain and each non-$S$-radical ideal is $S$-finite, then $R$ is an $S$-Noetherian ring.
	\noindent
	\begin{deff}$\cite{ed20}$\label{sintegral}
		A ring $R$ is said to be an $S$-integral domain if there exists $s \in S$ such that for all $a, b \in R$, whenever $ab = 0$, then either $sa = 0$ or $sb = 0$.
	\end{deff}
	\noindent
	Let $S_1\subseteq S_2$ be two multiplicatively closed subsets of $R$. Then every $S_1$-integral domain is also an $S_2$-integral domain. Following \cite{ed20}, we observe that every integral domain is an $S$-integral domain but the converse does not necessarily hold.
	
	\begin{eg}\label{sin}
		Take $R = \mathbb{Z}_{6}$, $S = \{\bar{1}, \bar{2}, \bar{4}\}$. Then $R$ is an $S$-integral domain but not an integral domain. To see this, let $a, b \in R \setminus\{\bar{0}\}$ with $ab = \bar{0}$. Then either $a = \bar{2}$ or $\bar{4}$ and $b = \bar{3}$. Setting $s = \bar{2} \in S$, we have $sb = \bar{0}$. Therefore, $R$ is an $S$-integral domain.
	\end{eg}
	\begin{eg}
		Let $R=\mathbb{Z}_{30}$, ~$S=\{\bar{1},\bar{2},\bar{4},\bar{8},\overline{16}\}$. It is clear that $R$ is not an $S$-integral domain, since $\bar{5} \cdot \bar{6} = \bar{0}$, but for any $s \in S$, neither $s \cdot \bar{5} = \bar{0}$ nor $s \cdot \bar{6} = \bar{0}$.
	\end{eg}
	
	\begin{deff}\label{cancel}
		A ring $R$ is said to have the $S$-cancellation property if for all $a,b,c\in R$ with $sa\neq 0$ for every $s\in S$ and $ab=ac$, then $sb=sc$ for some $s\in S$.
	\end{deff}
	With this definition, we can extend the cancellation property of the integral domain for the $S$-integral domain. In the following examples, we describe the $S$-cancellation property.
	\begin{eg}
		Let $R=\mathbb{Z}_{12}$, $S=\{\bar{1},\bar{3},\bar{9}\}$. For $a=\bar{2}$, observe that $sa\neq 0$ for all $s\in S$. Take $b=\bar{4}$ and $c=\overline{10}$ in $R$, then $ab=ac=\bar{8}$. Note that $sb\neq sc$ for all $s\in S$. Therefore $R$ does not have the $S$-cancellation property. 
	\end{eg}
	\begin{eg}
		Consider $R=\mathbb{Z}_{15}$, $S=\{\bar{1}, ~\bar{6}\}$. For $a, b, c\in R$, if $ab=ac$ and $sa\neq 0$ for all $s\in S$, then $sb=sc$ for $s=\bar{6}\in S$. In particular, for $a=\bar{3}, b=\bar{7}$ and $c=\bar{2}$, we have $sa\neq 0$ for all $s\in S$ and $ab=ac=\bar{6}$. This implies that $sb=sc=\overline{12}$ for $s=\bar{6}$. 
	\end{eg}
	\begin{prop}\label{scancellation}
		A ring $R$ is an $S$-integral domain if and only if the $S$-cancellation property holds in $R$.
	\end{prop}
	\begin{proof}
		Suppose $R$ is an $S$-integral domain. Let $a,b,c\in R$ be such that $sa\neq0$ for all $s\in S$ and $ab=ac$. Then, by definition of the $S$-integral domain, $sa=0$ or $s(b-c)=0$ for some $s\in S$. Also, $sa\neq 0 $ for all $s\in S$. Therefore $s(b-c)=0,$ i.e., $sb=sc$. Hence the $S$-cancellation property holds in $R.$ 
		
		Conversely, let $ab=0$. If $sa=0$ for some $s\in S,$ we are done. If $sa\neq 0$ for all $s\in S$, then $ab=0=a\cdot 0$, then by the definition of $S$-cancellation, $sb=s\cdot 0=0$ for some $s\in S$. Hence $R$ is an $S$-integral domain.
	\end{proof}
	
	Recall from \cite{ah20}, an ideal $P$ (which is disjoint with $S$) of $R$ is called an \textit{$S$-prime} ideal if there exists an $s \in S$ such that for any $a, b \in R$ with $ab \in P$, we have either $sa \in P$ or $sb \in P$. It is clear that every prime ideal is an $S$-prime, but converse does not hold in general (see \cite[Example 1(3)]{ah20}). 
	
	\begin{lem}$\cite{au24}$\label{int}
		A ring $R$ is an $S$-integral domain if and only if $(0)$ is an $S$-prime ideal.
	\end{lem}
	
	\begin{prop}$\cite{au24}$\label{sint}
		Let $P$ be an ideal of $R$ disjoint from $S$. Then $P$ is an $S$-prime ideal of $R$ if and only if $R/P$ is an $\overline{S}$-integral domain, where $\overline{S}=\{s+P\hspace{0.1cm}|\hspace{0.1cm}s\in S\}$ is a multiplicatively closed subset of $R/P$.
	\end{prop}
	
	Recall that $R$ is an integral domain if and only if $R[X]$ is an integral domain. We extend this result for $S$-Integral domain.
	\begin{prop}
		$R[X]$ is an $S$-integral domain if and only if $R$ is an $S$-integral domain.
	\end{prop}
	\begin{proof}
		Let $P$ be an ideal of $R$ which is disjoint with $S$, the fact that $(P[X]:_{R[X]} s)=(P:_R s)[X]$ and $R[X] / (P:_R s)[X] \cong (R/(P:_R s))[X]$ for some $s\in S$. According to \cite[Proposition 1]{ah20}, $P[X]$ is an $S$-prime ideal of $R[X]$ if and only if $(P[X]: _{R[X]} s)=(P:_R s)[X]$ is a prime ideal of $R[X]$, if and only if $(P:_R s)$ is a prime ideal of $R$ if and only if $P$ is an $S$-prime ideal of $R$. In particular if $P=(0)$, then $(0)$ is an $S$-prime ideal of $R$ if and only if $(0)$ is an $S$-prime ideal of $R[X]$. Thus $R[X]$ is an $S$-integral domain if and only if $R$ is $S$-integral domain.
	\end{proof}
	\begin{prop}\label{local}
		If a ring $R$ is an $S$-integral domain, then $S^{-1}R$ is an integral domain.
	\end{prop}
	\begin{proof}
		Suppose $R$ is an $S$-integral domain. Let $\dfrac{\alpha}{u_1}, \dfrac{\beta}{u_2} \in S^{-1}R$ be such that $\dfrac{\alpha}{u_1} \cdot \dfrac{\beta}{u_2} = 0$, where $\alpha, \beta \in R$ and $u_1, u_2 \in S$. Then there exists $s \in S$ such that $s\alpha\beta = 0$ in $R$. Since $R$ is an $S$-integral domain, there exists $u \in S$ such that either $us\alpha = 0$ or $u\beta = 0$. If $us\alpha= 0$, then $\dfrac{\alpha}{u_1} = \dfrac{us\alpha}{us u_1} = 0$. Similarly, if $u\beta = 0$, then $\dfrac{\beta}{u_2} = \dfrac{u\beta}{uu_2} = 0$. Therefore $S^{-1}R$ is an integral domain.
	\end{proof}
	
	Recall from \cite{ad09}, let $M$ be an $R$-module. The  idealization of $R$-module $M$, $R(+)M=\{(r, m)\mid r\in R, m\in M\}$ is a commutative ring with component wise addition and the multiplication defined by $(\alpha_1, m_1)(\alpha_2, m_2)= (\alpha_1\alpha_2, \alpha_1m_2 + \alpha_2m_1)$ for all $\alpha_1, \alpha_2 \in R$ and $m_1, m_2 \in M$. It is straightforward to verify that $S(+) M = \{(s, m) \mid s \in S, m \in M\}$ forms a multiplicative closed set in $R(+)M$.
	
	The following example shows that even if the localization $S^{-1}R$ is an integral domain, it does not guarantee that $R$ is an $S$-integral domain. 
	
	\begin{eg}\label{nsint}
		Let $\mathbb{E}=\bigoplus_{p\in\mathcal{P}}\mathbb{Z}/p\mathbb{Z}$, where $\mathcal{P}$ is the set of all prime numbers. Define $R=\mathbb{Z} (+) \mathbb{E}$ and $S=(\mathbb{Z}\setminus\{0\}) (+)(0)$. According to \cite[Example 3.12]{ki24}, the localization 
	$S^{-1}R=\mathbb{Q}$, which is a field and therefore is an integral domain. Contrary, suppose that $R$ is an $S$-integral domain, $(0)$ is an $S$-prime ideal of $R$. Consequently, there exists $s\in S$ such that for each $\alpha, \beta\in R$ if $\alpha\beta=0$, then either $s\alpha=0$ or $s\beta=0$. Take $s=(n, 0)$. Since for every $q\in \mathbb{E}$, $(0, q)^{2}= (0, 0)$, then there exists $(n, 0)\in S$ such that $(n, 0)(0, q)=(0,0)$. It follows that $nq=0$, for $n\in\mathbb{Z}\setminus\{0\}$. This implies $n\mathbb{E}=0$ which is not possible. Therefore $R$ is not an $S$-integral domain.
	\end{eg}
	\begin{theorem}\label{domain} Suppose  $S$ is finite subset of a ring $R$. Then
		 $R$ is an $S$-integral domain if and only if $S^{-1}R$ is an integral domain.
	\end{theorem}
	\begin{proof}
		If $R$ is an $S$-integral domain, then $S^{-1}R$ is an integral domain,  follows from Proposition \ref{local}. Conversely, let $a, b \in R$ be such that $ab=0$. Consequently, $\dfrac{ab}{s's''} = \dfrac{a}{s'}\cdot\dfrac{b}{s''} \in S^{-1}R$, where $s', s'' \in S$. Then $\dfrac{a}{s'} = 0$ or $\dfrac{b}{s''} = 0$ since $S^{-1}R$ is an integral domain. If $\dfrac{a}{s'}= 0$, then there exists $u_1 \in S$ such that $u_1a = 0$. Further, if $\dfrac{b}{s''} = 0$, then there exists $u_2 \in S$ such that $u_2b = 0$. Since $S$ is finite, define $u = \prod_{t\in S}t $, then $ua = 0$ or $ub = 0$. Hence $R$ is an $S$-integral domain. 
	\end{proof}
	
	\begin{eg}
		Consider $R=\mathbb{Z}\times \mathbb{Q}$, ~$S=\{(1,1),~ (0, 1)\}$. Evidently, $R$ is not an integral domain. Let $\alpha, ~\beta \in R$ such that $\alpha=(x, u)$ and $\beta=(y, v)$, where $x, y\in \mathbb{Z}$ and $u, v\in\mathbb{Q}$. Now, if $\alpha\cdot\beta=(xy, uv)=(0, 0)$, then it follows that $xy=0$ and $uv=0$. This implies that either $x=0$ or $y=0$ and either $u=0$ or $v=0$ since $\mathbb{Z}$ and $\mathbb{Q}$ are integral domains. Therefore there are two possibilities for $\alpha$ and $\beta$, if $\alpha=(x, 0)$, then $\beta=(0, v)$. Further, if $\alpha=(0, u)$, then $\beta=(y, 0)$. Take $s=(0,1)\in S$. Then either $s\alpha = 0$ or $s\beta = 0$. Thus $R$ is an $S$-integral domain, and hence by Theorem \ref{domain}, $S^{-1}R$ is an integral domain.
	\end{eg}

In $2024$, Ersoy et al. \cite{ba24}, introduced the concept of $S$-idempotent elements of a ring $R$. In this work, we identify the class of $S$-idempotent elements in an $S$-integral domain.
	\begin{deff}$\cite{ba24}$
		An element $a \in R$ is called \textit{$S$-idempotent} if $a^2 = s \cdot a$ for some $s \in S$. 
	\end{deff}
	Observe that every idempotent element is $S$-idempotent, but converse may not be true. 
	\begin{eg}\label{sidem}
		Let $R$ be any ring. We take $R\times \mathbb{Z}=\{(r, ~m)\mid r\in R, ~m\in \mathbb{Z}\}$. Then define addition coordinate wise and multiplication as follows: 
		$$(r, ~m)(s, ~n)=(rs+sm+nr, ~mn).$$ It is easy to verify that $R\times \mathbb{Z}$ is a commutative ring with identity $(0, ~1)$. Consider $S=\{0\}\times\mathbb{Z}^{*}$, where $\mathbb{Z}^{*}=\mathbb{Z}\setminus\{0\}$. Let $(0, ~a)\in R\times \mathbb{Z}$, where $a\in\mathbb{Z}\setminus\{0, 1\}$. Then $(0, ~a)^{2}=(0, ~a^{2})\neq (0, ~a)$ for $a^{2}\neq a$, that is, $(0, ~a)$ is not an idempotent element, but $(0, ~a)^{2}=s (0, ~a)$ for $s=(0, ~a)\in S$. Therefore $(0, ~a)$ is an $S$-idempotent element of~$R\times\mathbb{Z}$.
	\end{eg}
	\begin{deff}$\cite{ba24}$
		Let $y$ be an element of a ring $R$. Then we have the following:
		\begin{enumerate}
			\item If $sy^{n}=0$ for some $s\in S$ and $n\in \mathbb{N}$, then $y$ is called $S$-nilpotent.
			\item If $sy = 0$ for some $s\in S$, then $y$ is called $S$-zero. Note that $0$ is also an $S$-zero element.
		\end{enumerate}
	\end{deff}
    
    \begin{deff}
    An element $a \in R$ is said to be \emph{$S$-non-zero} if $sa \neq 0$ for all $s \in S$.
    \end{deff}
    
	\begin{deff}$\cite{ap20}$
		A ring $R$ is said to be $S$-reduced, if $r^{n}=0$, where $r \in R$, and $n \in\mathbb{N}$, then there exists $s\in S$ such that $sr= 0$.
	\end{deff}
	Notion of $S$-reduced ring \cite{ap20}, which is a proper generalization of reduced ring was first presented by A. Pekin et al. in $2020$. The $S$-version results for the $S$-integral domain is shown below.
	\begin{theorem}\label{si}
		Let $R$ be an $S$-integral domain. Then we have the following:
		\begin{enumerate}
			\item $R$ does not possess any $S$-non-zero $S$-nilpotent element.
			\item If $S\cap Z(R)=\emptyset$, then non-zero $S$-idempotent elements of $R$ must be elements of $S$. 
			\item $R$ is an $S$-reduced ring.
		\end{enumerate}
	\end{theorem}
	\begin{proof}
		\leavevmode
		\begin{enumerate}
			\item Suppose there exists $r\in R$ such that $sr\neq 0$ for all $s\in S$, i.e. $r$ is an $S$-non-zero and it is an $S$-nilpotent. Then there exists $t\in S$ such that $tr^{n}=0$ for some $n\in\mathbb{N}$. Since $R$ is an $S$-integral domain, there exists $s'\in S$ such that $s'r=0$, a contradiction, as $sr\neq 0$ for all $s\in S$. Hence the result.
			\item Let $a\in R$ be $S$-idempotent, so that $a^{2}=sa$ for some $s\in S$. This implies that $a(a-s)=0$, then there exists $t\in S$ such that either $ta=0$ or $t(a-s) =0$, since $R$ is an $S$-integral domain. Consequently, either $a=0$ or $a=s$ since $S\cap Z(R)=\emptyset $, as desired.
			\item Suppose $R$ is an $S$-integral domain. Let $a\in R$, and $a^{n}=0$ for some $n\in\mathbb{N}$. Then there exists $s'\in S$ such that $s'a=0$ since $R$ is an $S$-integral domain. Therefore $R$ is an $S$-reduced ring. 
		\end{enumerate}
	\end{proof}
	
	Note that the converse may not be true for Theorem \ref{si}(3).
	\begin{eg}
		Consider a ring $R=\mathbb{Z}_{p}\times \mathbb{Z}_{p}\times\cdots\times\mathbb{Z}_{p}\times\cdots$ (countably infinite copies of $\mathbb{Z}_{p})$, where $p$ is prime. Then $R$ is an $S$-reduced ring, for $R$ has no non-zero nilpotent elements. Consider a multiplicative set $S=\{{1_R}=(\bar{1},\bar{1},\bar{1},\ldots), s=(\bar{1},\bar{1},0,\ldots)\}$. Let $a=(\bar{1}, 0, \bar{1}, 0, \ldots)$ and ~$b=(0, \bar{1}, 0, \bar{1}, \ldots)\in R$. Then $ab=(\bar{0}, \bar{0}, \ldots)$, but neither $sa=(\bar{0}, \bar{0}, \ldots)$ nor $sb=(\bar{0}, \bar{0}, \ldots)$ for all $s\in S$. Therefore $R$ is not an $S$-integral domain. 
	\end{eg}

	In general, an integral domain need not be an $S$-Noetherian ring. We now present an example of an integral domain that is not $S$-Noetherian.
\begin{eg}\label{sn}
		Let $R=\mathbb{Z}+X\mathbb{Q}[X]$ and $S=\{+1, -1\}$ be multiplicatively closed set of $R$. Clearly, $R$ is an integral domain. Consider \(R[Y]\) which is not Noetherian so is not $S$-Noetherian.
\end{eg}
	This leads to a natural question:
	\begin{question}
		When is an integral domain, an $S$-Noetherian ring?
	\end{question}
	The answer to the above question is given in Theorem \ref{srad} using $S$-radical ideals.
	
	\begin{deff}\cite[Definition 2.1]{ut25}
		Let $S\subseteq R$ be a multiplicatively closed set, and $I$ be an ideal of $R$. Then $S$-radical of $I$ is defined by
		$$\sqrt[S]{I} = \{a \in R \mid sa^n \in I \text{ for some } s \in S \text{ and } n \in \mathbb{N}\}.
		$$ Also, if $\sqrt[S]{I} =I$, then $I$ is said to be an $S$-radical ideal of $R$.
	\end{deff}
	\begin{theorem}\label{srad}
		If every non-$S$-radical ideal of an integral domain $R$ is $S$-finite, then $R$ is an $S$-Noetherian ring.
	\end{theorem}
	\begin{proof}
		We show that every proper ideal of $R$ is $S$-finite. If $I\cap S\neq\emptyset$, then there exists $t\in S$ such that $t\in I$, then $tI\subseteq J\subseteq I$, where $J=<t>$ is the ideal generated by $t$. Hence $I$ is an $S$-finite. Assume that $I\neq (0)$ and $I\cap S=\emptyset$. For $s\in S$, there exists $0\neq x_{s}\in I$ such that $sx_{s}\neq 0$. As $x_{s}\in I$, but $sx_{s}\notin x_{s}I$, for this, if $sx_{s}\in x_{s}I$, then $sx_{s}=x_{s}i$ for some $i\in I$ but by the cancellation property this implies $s=i\in I$, which is not possible since $I\cap S=\emptyset$. Thus $sx_{s}\notin x_{s}I$, but $sx_{s}^{2}\in sx_{s}I\subseteq x_{s}I$. This implies that $x_{s}I$ is a non-$S$-radical, so $x_{s}I$ is an $S$-finite. Now, define an $R$-module homomorphism $\phi_{s} :I\rightarrow x_{s}I$ by $\phi_{s}(j)=x_{s}j$ for all $j\in I$. Then $$ker~\phi_{s}=\{y\in I\hspace{0.1cm}|\hspace{0.1cm}\phi_{s}(y)=x_{s}y=0\} =\{y\in I\hspace{0.1cm}| x_{s}=0 \hspace{0.1cm}or \hspace{0.1cm} y=0 \hspace{0.1cm} \}.$$
		Since $R$ is an integral domain, $ker~\phi_{s}=\{0\}$. Hence $\phi_{s}$ is injective. Next, if $a\in x_{s}I$, then $a=x_{s}i$ for some $i\in I$ such that $\phi_{s}(i)=x_{s}i=a$. Thus $\phi_{s}$ is surjective, and hence $\phi_{s}$ is an $R$-module isomorphism, i.e., $x_{s}I\cong I$ as $R$-module. Therefore $I$ is an $S$-finite, and hence $R$ is an $S$-Noetherian ring.
	\end{proof}
In the context of Theorem \ref{srad}, the following question arises, which we leave open for future study.
   
	\begin{question}
		Suppose every non-$S$-radical ideal of an $S$-integral domain is $S$-finite. Is $R$ $S$-Noetherian?
	\end{question}

\section{Characterization of $S$-field}	\label{sec}
In this section, we extend the concept of a field to an $S$-field. We introduce the notion of $S$-proper ideals in $R$ and establish results analogous to those in fields. First, we recall definitions of $S$-maximal ideal and $S$-field from \cite{ed20,ed22}, and provide several examples.
	\begin{deff}\cite[Definition 2.2]{ed22}\label{smax}
		An ideal $I\subseteq R$ (disjoint from $S$) is called  an $S$-maximal ideal if there exists $s\in S$ such that for any ideal $J$ of $R$ with $I\subseteq J$, then either $sJ \subseteq I$ or $J\cap S\neq\emptyset$.
	\end{deff}
	It is evident that every maximal ideal of $R$ qualifies as an $S$-maximal ideal for any multiplicatively closed set $S$ in $R$. However, the converse does not hold.
	\begin{eg}\label{eg}
		Consider $R=\mathbb{Z}_{6}$, $S=\{\bar{1}, \bar{2}, \bar{4}\}$. The zero ideal $(\bar{0})$ is not a maximal ideal in $R$. The proper ideals of $R$ are $I_1=2\mathbb{Z}_{6}$ and $I_2=3\mathbb{Z}_{6}$. Since zero ideal $(\bar{0})\subseteq I_1$ and $I_{1}\cap S\neq\emptyset$, and $(\bar{0})\subseteq I_2$, then for $s=\bar{2}\in S,$ $sI_{2}\subseteq (\bar{0})$. Then $(\bar{0})$ is an $S$-maximal ideal.
	\end{eg}
	\begin{deff}\cite[Definition 9]{ed20}\label{sfield}
		A ring $R$ is said to be an $S$-field if the zero ideal $(0)$ is an $S$-maximal ideal of $R$.
	\end{deff}
	\noindent
	Note that every $S$-field is an $S$-integral domain, and every field is an $S$-field for any multiplicatively closed set $S\subseteq R$ but the converse is not true.
	
	\begin{eg}
		Let $R=\mathbb{Z}$, $S=\{2^{i}|\hspace{0.1cm} i\in\mathbb{N}\}$. Then $R$ is an $S$-integral domain since $R$ is an integral domain. But $R$ is not an $S$-field. For this, $(0)\subseteq I$, where $I=3\mathbb{Z}$ is an ideal of $R.$ Observe that $I\cap S=\emptyset$ and there does not exist any $s\in S$ such that $sI\subseteq (0)$. 
	\end{eg}
	
	\begin{eg}\label{infi}
		Let $R = \mathbb{Z}$, $S = \mathbb{Z} \setminus \{0\}$. Then for any ideal $I$ of $R$, $I \cap S \neq \emptyset$ and the ideal $(0)$ is an $S$-maximal. Hence $R$ is an $S$-field.
	\end{eg}
	
	\begin{eg}\label{zmn}
		Consider the ring $R=\mathbb{Z}_{mn}$, and let $S=\{\bar{m}^{i}\hspace{0.1cm}|\hspace{0.1cm}i\in\mathbb{N}\cup\{0\}\}$ be a multiplicatively closed subset of $R$, where $m$, $n$ are distinct prime numbers. The proper ideals of $R$ are $I_1=m\mathbb{Z}_{mn}$ and $I_2=n\mathbb{Z}_{mn}$. As zero ideal $(\bar{0})\subseteq I_1$ and $I_{1}\cap S\neq\emptyset$, and $(\bar{0})\subseteq I_2$, then for $s=\bar{m}\in S,$ $sI_{2}\subseteq (\bar{0})$. Then $(\bar{0})$ is an $S$-maximal ideal but not a maximal. Thus $R$ is an $S$-field but not a field.
	\end{eg}
	
	\begin{deff}\label{sprop}
		An ideal $I\subseteq R$ is said to be an $S$-proper ideal if $I\cap S = \emptyset$ and $sI\neq 0$ for all $s\in S.$ 
	\end{deff} 
	Observe that every $S$-proper ideal is a proper ideal, but a proper ideal of $R$ need not be an $S$-proper ideal. For instance, in Example \ref{zmn}, the ideal $I_1=m\mathbb{Z}_{mn}$ is a proper ideal but not an $S$-proper ideal since $S\cap I_1\neq \emptyset$.
	
	\begin{prop}\label{nosproper}
		A ring $R$ is an $S$-field if and only if it has no $S$-proper ideal.
	\end{prop}
	\begin{proof}
		Assume $R$ is an $S$-field and $I$ is a non-zero ideal of $R$. Then $(0)$ is an $S$-maximal ideal and $(0)\subseteq I,$ therefore, $I \cap S\neq \emptyset$ or $sI\subseteq (0)$ for some $s\in S.$ Thus $I$ is not $S$-proper. Conversely, suppose $R$ has no $S$-proper ideals. If $I$ is not an $S$-proper ideal, then by definition, $I\cap S\neq \emptyset$ or $sI=(0)$ for some $s\in S$. Thus $(0)$ is an $S$-maximal ideal.
	\end{proof}
	
	\begin{prop}\label{mfield}
		
		An ideal $M\subseteq R$ such that $M\cap S=\emptyset$ is an $S$-maximal if and only if $R/M$ is an $\overline{S}$-field, where $\overline{S}=\{s+M\hspace{0.1cm}|\hspace{0.1cm}s\in S\}$ is a multiplicatively closed subset of $R/M$. 
	\end{prop}
	\begin{proof}
		Suppose $M$ is an $S$-maximal ideal of $R$. Then, by Definition \ref{sfield}, $R/M$ is an $\overline{S}$-field. 
		
		Conversely, assume that $R/M$ is an $\overline{S}$-field. Let $M\subseteq M'$ for some ideal $M'$ of $R$. As $M'/M$ is an ideal of $R/M$ containing $M$. By Proposition \ref{nosproper}, $M'/M$ is not an $\overline{S}$-proper ideal of $R/M$. Then there are two cases:\\
		\textbf{Case 1:} If $(M'/M)\cap \overline{S}\neq\emptyset$, then there exists $\bar{s}\in\overline{S}$ such that $\bar{s}\in M'/M$. This implies that $s+M\in M'/M$, $s\in M'$ since $M\cap S=\emptyset$. Thus $ M'\cap S\neq\emptyset$, and hence $M$ is an $S$-maximal ideal of $R$.\\
		\textbf{Case 2:} If $\bar{s}M'/M\subseteq M$, the zero ideal of $R/M$ for some $\bar{s}\in\overline{S}$. This implies that $(s+M) M'/M\subseteq M$, $sM'/M\subseteq M$. Thus $ sM'\subseteq M$, and therefore $M$ is an $S$-maximal ideal of $R$.	
	\end{proof}
	\begin{cor}
		If $\phi$ is a surjective homomorphism from a ring $R$ to an $\overline{S}$-field $F$, then $ker\phi$ is an $S$-maximal ideal of $R$.
	\end{cor}
	A ring $R$ with $x^2=x$ for all $x\in R$ is called a Boolean ring. In Boolean rings, every prime ideal is maximal. We extend this result to $S$-version in the next Proposition.
	\begin{prop}
		Every $S$-prime ideal is an $S$-maximal ideal in a Boolean ring $R$.
	\end{prop}
	\begin{proof}
		Let $P\subseteq R$ be an $S$-prime ideal and $P\subseteq Q$ with $Q\cap S=\emptyset$, where $Q$ is an ideal of $R$. We show that $sQ\subseteq P$ for some $s\in S$. Let $x\in Q$. Then $x^{2}r=xr$ for all $r\in R$. Consequently, $x(xr-r)=0\in P$, then there exists $s\in S$ such that $sx\in P$ or $s(xr-r)\in P$. If $s(xr-r)\in P$, $sxr-sr=p$ for some $p\in P$. This implies that $sr=sxr-p\in Q$ since $x\in Q$ and $P\subseteq Q$. Consequently, $sR\subseteq Q$. Then $s\in Q$, which is not possible since $R$ has unity and $S\cap Q=\emptyset$. Thus $sx\in P$, $sQ\subseteq P$. Hence $P$ is an $S$-maximal ideal of $R$.
	\end{proof}
	\begin{cor}
		Let $R$ be a Boolean ring, and $P$ be an $S$-prime ideal of $R$. Then $R/P$ is an $\bar{S}$-field, where $\bar{S}=\{s+P\mid s\in S\}$ is a multiplicatively closed set of $R/P$.
	\end{cor}
	
	\begin{theorem}\label{uni}
		Let $R$ be an $S$-Artinian ring, and $P$ be an $S$-prime ideal of $R$. Then $R/P$ is an $\bar{S}$-field, where $\bar{S}=\{s+P\mid s\in S\}$ is a multiplicatively closed subset of $R/P$.
	\end{theorem}
	\begin{proof}
		Let $R$ be an $S$-Artinian ring, and $P$ be an $S$-prime ideal of $R$. Then $R/P$ is $\overline{S}$-integral domain. To show $R$ is an $S$-field, it is enough to show that $P$ is an $S$-maximal ideal of $R$. Suppose $P\subseteq Q$ for some ideal $Q$ of $R$ with $sQ\nsubseteq P$ for all $s\in S$. Then for all $s\in S$, there exists $q\in Q$ such that $sq\notin P$. Since $P$ is  $S$-prime, there exists $t\in S$ such that $(P:t)$ is a prime ideal of $R$, by \cite[Proposition 1]{ah20}. According to our assumption, for this $t$, there exists $q'\in Q$ such that $tq'\notin P$. This implies that $q'\notin P$. Our aim is to show that $Q\cap S\neq\emptyset$. Consider the following decreasing sequence of ideals of $R$
		$$P+Rq'\supseteq P+Rq'^{2}\supseteq\cdots\supseteq P+Rq'^{n}\supseteq \cdots .$$
		Since $R$ is an $S$-Artinian, there exist $s'\in S$ and $k\in \mathbb{N}$ such that $s'(P+Rq'^{k})\subseteq (P+Rq'^{k+1})$. Let $p\in P$. Then $s'(p+q'^{k})=p'+q'^{k+1}r$ for some $r\in R$ and $p'\in P$. This implies that $q'^{k}(s'-q'r)\in P$. Consequently, $q'^{k}(s'-q'r)\in (P:t)$. Then either $q'^{k}\in (P:t)$ or $s'-q'r\in (P:t)$ since $(P:t)$ is a prime ideal of $R$. If $q'^{k}\in (P:t)$, $q'\in (P:t)$. Thus $tq'\in P$, which contradicts the assumption that $tq'\notin P$. On the other hand, if $s'-q'r\in (P:t)$, then $ts'-trq'\in P\subseteq Q$. Since $trq'\in Q$, this implies that $ts'\in Q$. Thus $Q\cap S\neq\emptyset$, as desired.
	\end{proof}
	Next we investigate $S$-integral domains and $S$-fields under ring homomorphisms. Let $f: R\longrightarrow R'$ be a ring homomorphism and $S$ be a multiplicatively closed subset of $R$. Then $f(S)$ is a multiplicatively closed subset of $R'$ if and only if $S\cap \text{Ker}(f)=\emptyset$. In particular, if $f$ is one-one, then $f(S)$ is a multiplicatively closed set in $R'$.
	\begin{prop}
		Let $f$ be an isomorphism of a ring $R$ onto a ring $R'$. Then we have the following:
		
		\begin{enumerate}
			\item If $R$ is an $S$-integral domain, then $R'$ is an $f(S)$-integral domain.
			\item If $R$ is an $S$-field, then $R'$ is an $f(S)$-field.
		\end{enumerate}
	\end{prop}
	\begin{proof}
		\leavevmode
		\begin{enumerate}
			\item Let $R$ be an $S$-integral domain. Since $S\subseteq R$ is a multiplicatively closed set in $R$ and $f$ is injective, $0\notin f(S)$, as $0\notin S$ and $f(1_{R})=1_{R'}\in f(S)$. Thus $f(S)$ is a multiplicatively closed subset of $R'$. Let $\alpha, ~\beta \in R'$ be such that $\alpha\cdot \beta=0_{R'}$. Since $f$ is an isomorphism, there exist unique $a, b\in R$ such that $f(a)=\alpha$ and $f(b)=\beta$. Now $\alpha\cdot \beta=0_{R'}$ implies $f(a)f(b)=0_{R'}$. Consequently, $f(ab)=f(0_{R})$. This implies that $ab=0_{R}$. Since $R$ is an $S$-integral domain, there exists $s\in S$ such that either $sa=0$ or $sb=0$. It follows that either $f(s)\alpha=0_{R'}$ or $f(s)\beta=0_{R'}$. Thus $R'$ is an $f(S)$-integral domain.
			\item Let $R$ be an $S$-field. Our aim is to show that the zero ideal $(0_{R'})\subseteq R'$ is an $f(S)$-maximal ideal of $R'$. Clearly, $(0_R')\subseteq I_{R'}$ for some ideal $I_{R'}\subseteq R'$. Let $J_R=f^{-1}(I_{R'}).$ Then $J_R$ is an ideal of $R$ and contains zero ideal $(0_R)$ of $R.$ Therefore, there exists an $s\in S$ such that either $sJ_{R}\subseteq (0_{R})$ or $S\cap J_{R}\neq \emptyset$. Consequently, we have $f(s)I_{R'}=f(sJ_{R})\subseteq f(0_{R})=(0_{R'})$ or $\emptyset\neq f(S)\cap f(J_{R})=f(S)\cap I_{R'}$ since $f$ is an isomorphism. Therefore $(0_{R'})$ is an $S$-maximal ideal of $R'$. Hence $R'$ is an $f(S)$-field.
		\end{enumerate}
	\end{proof}
 \begin{deff}
        A multiplicatively closed set is said to be proper if it does not contain zero or zero divisors of $R$. 
    \end{deff}
    
\begin{rem}\label{mul}
Let $S$ be a multiplicatively closed subset of $R$.	Define $\phi: R \to S^{-1}R$ as $\phi(r)=\dfrac{r}{1}$ for all $r\in R$. Then $\phi(S) =\left \{\dfrac{s}{1}\in S^{-1}R |\ s\in S\right\}$ is a multiplicatively closed suset of $S^{-1}R$. Moreover, observe that if  $S$ is a proper multiplicatively closed set, then $\phi(S)$ is also a proper multiplicatively closed set.
	\end{rem}

	\begin{prop} \label{S-1rfield}
		If $R$ is an $S$-field, then $S^{-1}R$ is a $\phi(S)$-field.
	\end{prop}
	\begin{proof}
		We denote zero ideal in $S^{-1}R$ by $[0]$ and zero ideal in $R$ by $(0)$. By Remark \ref{mul}, $\phi(S)$ is a multiplicatively closed set of $S^{-1}R.$ Let $R$ be an $S$-field. We show that $S^{-1}R$ is a $\phi(S)$-field, that is, $[0]$ is a $\phi(S)$-maximal ideal in $S^{-1}R$. Suppose that $[0]\subseteq S^{-1} I\subseteq S^{-1}R$, for some ideal $I$ of $R$. 
		Define $J=\{a\in R| \ \phi(a)\in S^{-1}I\}.$ Note that whenever $\dfrac{a}{s}\in S^{-1}I$, then $\dfrac{s}{1}\cdot \dfrac{a}{s}=\dfrac{a}{1}\in S^{-1}I,$ that is, $a\in J.$ Thus $J=\{a\in R| \ \dfrac{a}{s}\in S^{-1}I, \text{ for some } s\in S\}$. It is easy to see that $J$ is an ideal of $R$ and $(0)\subseteq J$. Since $R$ is an $S$-field, then either $J\cap S\neq \emptyset$ or there is an $s\in S$ such that $sJ\subseteq(0)$.
		\begin{description}
			\item[Case 1:]Let $a\in J\cap S\neq \emptyset$, then $\dfrac{a}{s}\in S^{-1}I$ for some $s\in S$. Since $S^{-1}I$ is an ideal therefore, $\dfrac{s}{1}\cdot\dfrac{a}{s}=\dfrac{a}{1}\in S^{-1}I$. Thus $\dfrac{a}{1}\in~S^{-1}I\cap\phi(S)\neq \emptyset$. Thus $[0]$ is a $\phi(S)$-maximal ideal of $S^{-1}R$.
			\item[Case 2:] If $J\cap S=\emptyset$, then $sJ\subseteq(0)$. Consequently, for any $\dfrac{a}{s'}\in S^{-1}I,$ we have $\phi(s)\cdot\dfrac{a}{s'}=\dfrac{s}{1}\cdot\dfrac{a}{s'}=\dfrac{0}{s'}=\dfrac{0}{1},$ that is, $\phi(s)S^{-1}I\subseteq [0]$. Therefore $[0]$ is a $\phi(S)$-maximal ideal.\\ Hence, $S^{-1}R$ is a $\phi(S)$-field.
		\end{description}
	\end{proof}
	The converse of the above theorem is not true in general. 
	\begin{eg}\label{sexample}
		Let $\mathbb{E}=\bigoplus_{p\in\mathcal{P}}\mathbb{Z}/p\mathbb{Z}$, where $\mathcal{P}$ is the set of all prime numbers. Define $R=\mathbb{Z}(+)\mathbb{E}$ and $S= (\mathbb{Z}\setminus\{0\}) (+)(0)$. By \cite[Example 3.12]{ki24}, the localization 
	$S^{-1}R=\mathbb{Q}$, which is a field, but $R$ is not an $S$-field since $R$ is not an $S$-integral domain (see Example \ref{nsint}).
	\end{eg}
	Next, we give some sufficient conditions on multiplicatively closed sets so that the converse of  Proposition \ref{S-1rfield} holds.
	\begin{theorem}\label{S-1r}
		Let $S$ be a proper multiplicatively closed subset of $R$, and $S^{-1}R$ is a $\phi(S)$-field. Then $R$ is an $S$-field if any one of the following holds
		\begin{enumerate}
			\item $S$ is a finite set.
			\item $R$ is $S$-PID.
		\end{enumerate}
	\end{theorem}
	\begin{proof}
		\begin{enumerate}
			\item Let $J$ be an ideal of $R$. Then $S^{-1}J\subseteq S^{-1}R$ is an ideal containing the zero ideal $[0]$ of $S^{-1}R.$ Since $S^{-1}R$ is a $\phi(S)$-field, therefore, $S^{-1}J\cap \phi(S)\neq \emptyset $ or there exists $\dfrac{t}{1}\in \phi(S)$ such that $\dfrac{t}{1}(S^{-1}J)\subseteq [0].$
			\begin{description}
				\item[Case 1:] Suppose $S^{-1}J\cap \phi(S)\neq \emptyset$ and $\dfrac{a}{s}\in S^{-1}J\cap \phi(S).$ This implies that $\dfrac{a}{s}=\dfrac{j}{s_1}=\dfrac{s_2}{1}$ for some $s_1,s_2\in S$ and $j\in J.$ Consequently, there exists $u\in S$ such that $u(j-s_1s_2)=0$. Thus $j=s_1s_2$ since $S$ is proper. Thus $s_1s_2\in J\cap S\neq \emptyset$.
				\item[Case 2:] Suppose $\dfrac{t}{1}S^{-1}J\subseteq [0],$ where $\dfrac{t}{1}\in \phi(S)$. This implies that $S^{-1}J\subseteq [0]$, since $\dfrac{t}{1}$ is a unit in $S^{-1}R.$ Consequently, $\dfrac{j}{s}=\dfrac{0}{1}$ for all $s\in S$ and $j\in J.$ In particular, for $s=1,$ $\dfrac{j}{1}=\dfrac{0}{1}$, $\forall j\in J$ which implies that there exists $s_j\in S$ for all $j$ such that $s_j j=0.$ Take $u=\prod_{s\in S} s$, then $uJ\subseteq (0),$ because $u\in S$ is finite product of elements in $S.$
			\end{description}
			Thus $(0)$ is an $S$-maximal that is $R$ is an $S$-field.
			\item Let $(0)\subseteq J$ be an ideal of $R.$ Since $R$ is an $S$-PID, there exist $s\in S$ and $a\in J$ such that $sJ\subseteq \langle a\rangle.$ Thus to show $R$ is an $S$-field, it is sufficient to show that whenever $(0)\subseteq \langle a\rangle$, then $\langle a\rangle \cap S\neq \emptyset$ or there exists $t\in S$ such that $t\langle a\rangle\subseteq (0).$ 
			Let $S^{-1}R$ be $\phi(S)$-field, and $(0)\subseteq \langle a\rangle,$ where $a\in R$. Then $S^{-1}\langle a \rangle$ is an ideal of $S^{-1}R$ containing the zero ideal $[0]$ of $S^{-1}R.$ Therefore, $S^{-1}\langle a\rangle\cap\phi(S)\neq \emptyset$ or there exists $\dfrac{t}{1}\in \phi(S)$ such that $\dfrac{t}{1}S^{-1}\langle a\rangle \subseteq [0].$
			\begin{description}
				\item[Case 1:] Let $S^{-1}\langle a\rangle \cap \phi(S)\neq \emptyset$ and $\dfrac{t}{s}\in S^{-1}\langle a\rangle \cap \phi(S).$ This implies that $\dfrac{t}{s}=\dfrac{b}{s_1}=\dfrac{s_2}{1}$ for some $b\in\langle a\rangle $ and $s_1,s_2\in S.$ Consequently, $u(b-s_1s_2)=0$ for some $u\in S$. Thus $b=s_1s_2,$ since $S$ is proper. Hence $b\in \langle a\rangle \cap S\neq \emptyset.$
				\item[Case 2:] Let $\dfrac{t}{1}(S^{-1}\langle a\rangle)\subseteq [0]$. This implies that $S^{-1}\langle a\rangle \subseteq [0],$ since $\dfrac{t}{1}$ is a unit in $S^{-1}R$. Consequently, $s\langle a\rangle\subseteq (0)$ for some $s\in S$.
			\end{description}
			Thus $(0)$ is an $S$-maximal that is $R$ is $S$-field.
		\end{enumerate}
	\end{proof}
	Every finite integral domain is a field. A natural question is: Does analogous result hold for the $S$-integral domain? The answer to this question is affirmative.
	
	\begin{cor}\label{cor}
		Every finite $S$-integral domain is an $S$-field.
	\end{cor}
	\begin{proof}
		Let $R$ be finite $S$-integral domain. Then, by Theorem \ref{domain}, $S^{-1}R$ is a finite integral domain because $S$ is finite. Therefore $S^{-1}R$ is a field since every finite integral domain is a field. As every field is an $S$-field for any multiplicatively closed set $S,$ therefore, $S^{-1}R$ is a $\phi(S)$-field. Also $S$ is finite; then by Theorem \ref{S-1r}, $R$ is an $S$-field.
	\end{proof}
	\section{$\mathbf{S}$-Version of Krull intersection theorem}
	In 2024, Kim et al. \cite{ki24} gave the definition of $S$-dimension and prove that if $R$ is an $S$-Noetherian $S$-domain with $dim_{S}(R)=1$. Then every proper ideal $S^{-1}I$ of $S^{-1}R$, where $I$ is an ideal of $R$ disjoint with $S$ can be uniquely written as a product of primary ideals whose radicals are distinct (see Theorem \ref{produt}). Krull in the 1930, provided a fundamental insight into the structure of Noetherian rings. He proved that for a Noetherian ring $R$ and an ideal $I$, we have $\bigcap\limits_{n=1}^{\infty}I^{n}=I(\bigcap\limits_{n=1}^{\infty}I^{n})$. Moreover, if $R$ is an integral domain, then $\bigcap\limits_{n=1}^{\infty}I^{n}=0$ (see \cite{ad75}). In 2024, Singh et al. \cite{ts24} proved the existence of $S$-primary decomposition in $S$-Noetherian rings. In this section, we prove the $S$-Krull intersection theorem using $S$-primary decomposition. 
	\begin{deff}$\cite{ki24}$
		Let $R$ be a ring and $S \subseteq R$ be a multiplicatively closed set. 
		\begin{enumerate}
			\item Then the sequence $P = P_{0} \supset P_1\supset\cdots \supset P_{n-1} \supset P_{n}$ is said to be an $S$-strictly decreasing chain (or an $S$-chain, for short) if for every $i\in\{0, \ldots, n-1\}$ and for each $s\in S$, we have $sP_{i+1}\nsubseteq P_i$.
			\item 	Let $P$ be an $S$-prime ideal of $R$. Then $S$-height of $P$, denoted by $S$-$ht(P)$, is the supremum of the lengths $n$ of all $S$-strictly decreasing chains $P = P_{0} \supset P_1\supset\cdots \supset P_{n-1} \supset P_{n}$ of $S$-prime ideals of $R$, and $dim_{S}(R) = sup\{$S$-ht(P)\mid P\in Spec_{S}(R)\}$ is called the $S$-Krull dimension of $R$.
		\end{enumerate}
	\end{deff}
	\begin{rem}\label{sp}
		If $P_{i}$ $(1\leq i\leq n)$ are distinct $S$-prime ideals, then $S^{-1}P_{i}$ need not be distinct. For this, consider $R=\mathbb{Z}_{12}$, ~$S=\{\bar{1}, \bar{2},\bar{4}, \bar{8}\}$ and let $P_{1}=3\mathbb{Z}_{12}$, $P_{2}=6\mathbb{Z}_{12}$ be two ideals of $R$. Observe that $P_{i}\cap S=\emptyset$ for $i=1, 2$. Since $P_{1}$ is a prime ideal of $R$, and so $S$-prime ideal of $R$. Clearly, $P_{2}$ is also an $S$-prime ideal of $R$. To show this let $a, b\in R$ be such that $ab\in P_{2}$. Then choice of $a=\overline{2k}$ and $b=\overline{3k'}$ or ($a=\overline{3k'}$ and $b=\overline{2k}$), where $k, ~k'\in\mathbb{Z}$, then it follows that $sb\in P_{2}$ or ~$sa\in P_{2}$~ for $s=\bar{4}$. Thus $P_{2}$ is an $S$-prime ideal of $R$. Evidently, $S^{-1}P_{1}=(\bar{0})=S^{-1}P_{2}$, but $P_{1}\neq P_{2}$.
	\end{rem}
\noindent
Recall from \cite[Proposition 2.5]{me22} that if $Q$ is an $S$-primary ideal of a ring $R$, then $P=\text{rad}(Q)$ is an $S$-prime ideal. In such a case,  $Q$ is called  $P$-$S$-primary ideal of $R$.
	\begin{theorem}\label{produt}
		Let $R$ be an $S$-Noetherian $S$-domain with $dim_{S}(R)=1$. Then every proper ideal $S^{-1}I$ of  $S^{-1}R$, where $I$ is an ideal of $R$ disjoint with $S$ can be uniquely written as a product of primary ideals whose radicals are all distinct.
	\end{theorem}
	\begin{proof}
		Since $R$ is $S$-Noetherian, $I$ has a minimal $S$-primary decomposition $I=\bigcap\limits_{i=1}^{n}Q_{i}$, by \cite[Theorem 9]{ts24}, where each $Q_{i}$ is $S$-primary and $rad(Q_{i})=P_{i}$. Consequently, $S^{-1}I= \bigcap\limits_{i=1}^{n}S^{-1}Q_{i}$, where each $S^{-1}Q_{i}$ is a primary ideal and $rad(S^{-1}Q_{i})=S^{-1}P_{i}$, by \cite[Remark 11]{ts24}. It is clear from Remark \ref{sp}, all $S^{-1}P_{i}$ are not necessarily distinct. Assume that if $rad(S^{-1}Q_{{i}})$ are equal for $i=1,\ldots, k$, say, $S^{-1}P$, and $rad(S^{-1}Q_{{i}})$ are all distinct for $i=k+1,\ldots, n$, then by \cite[Lemma 4.3]{fm69}, $S^{-1}I'=\bigcap\limits_{j=1}^{k}S^{-1}Q_{{j}}$ is also $S^{-1}P$-primary, so we replace $S^{-1}Q_{{1}}, ~S^{-1}Q_{{2}}, \ldots, S^{-1}Q_{{k}}$ by $S^{-1}I'$ in the decomposition. Now we can guarantee that $S^{-1}P_i\neq S^{-1}P_{j}$ for $i \neq j$ for $i, j\in\{1, 2, \ldots, n\}$ and $ \bigcap_{j\in\{1,2,\ldots,n\}\setminus\{i\}}S^{-1}Q_{j}\nsubseteq S^{-1}Q_{i}$. Thus $S^{-1}I$ represents the minimal primary decomposition. Since $R$ is $S$-domain and $dim_{S}(R)=dim(S^{-1}R)=1$, by \cite[Theorem 4.11]{ki24}, each non-zero prime ideal of $S^{-1}R$ is maximal, hence the $S^{-1}P_{i}$ are distinct maximal ideals, and are therefore pairwise coprime. Hence, by \cite[Proposition 1.16]{fm69} the $S^{-1}Q_{i}$ are pairwise comaximal  and therefore by \cite[Proposition 1.10]{fm69} we have $\prod_{i=1}^{n}S^{-1}Q_{i}=\bigcap\limits_{i=1}^{n}S^{-1}Q_{i}=S^{-1}I$.
		On the other hand, if $S^{-1}I=\prod_{i=1}^{n}S^{-1}Q_{i}$, then the same arguments lead to $S^{-1}I=\bigcap\limits_{i=1}^{n}S^{-1}Q_{i}$; the minimal primary decomposition of $S^{-1}I$, where $S^{-1}Q_{i}$ is an isolated primary component and unique by \cite[Corollary 4.11]{fm69}.
	\end{proof}
	To prove the $S$-version of the Krull intersection theorem, we need the following lemma. 
	\begin{lem}\label{pmsb}
		Let $R$ be an $S$-Noetherian ring and $I$ be an ideal disjoint with $S$. Then there exist $s\in S$ and an integer $m$ such that $s(rad(I))^m\subseteq I$, where $rad(I)=\{x\in R| \ x^n\in I \text{ for some } n\in \mathbb{N}\}$ is the radical ideal of $I$.
	\end{lem}
	\begin{proof}
		Since $rad(I)$ is an $S$-finite, $t\cdot rad(I)\subseteq J\subseteq rad(I)$ for some $t\in S$ and a finitely generated ideal $J=\langle x_1,x_2,\dots,x_l\rangle$ of $R$. Suppose $n_i\in \mathbb{N}$ be such that $x_i^{n_i}\in I$. Take $m=\sum_{i=1}^l (n_i-1)+1$. Then $J^m\subseteq I$. Consequently, $s(rad(I))^m\subseteq J^m\subseteq I$, where $s=t^m$. 
	\end{proof}
	\begin{theorem}(\textbf{$S$-Krull intersection theorem})\label{krull}
		Let $R$ be an $S$-Noetherian domain and $I$ be an ideal of $R$ disjoint from $S$. Let $B=\bigcap\limits_{n=1}^{\infty}I^{n}$. Then there exists $t\in S$ such that $tB=0$.
	\end{theorem}
	\begin{proof}
		
		Evidently, $IB\cap S=\emptyset$ because $I\cap S=\emptyset$. Then $IB$ admits $S$-primary decomposition since in $S$-Noetherian, any ideal which is disjoint from $S$ admits an $S$-primary decomposition \cite{st23}. Write $$IB=B_{1}\cap B_{2}\cap\cdots \cap B_{k},$$ where $B_{i}$ $(1\leq i\leq k)$ is $P_{i}$-$S$-primary ideal of $R$, where $P_i=rad(B_i)$ is an $S$-prime. Since $rad(B_{i})$ is $S$-finite, therefore $s_i(rad(B_{i}))\subseteq J_i\subseteq rad(B_{i})$ for some finitely generated ideal $J_i$ of $R$ and $s_i\in S$. 
		
		By Lemma \ref{pmsb}, there exist $s_i$ and integer $m_i$ such that $s_iP_i^{m_i}\subseteq B_i.$ Take $s=\prod_{i=1}^ks_i\in S$ and $m=\text{max}_{1\leq i\leq k}\{m_i\}$. Then $sP_i^m\subseteq B_i$ for all $1\leq i\leq k$.
		Observe that $IB\subseteq B_i$ for all $1\leq i\leq k$. As $B_i$ is $P_i$-$S$-primary, there exists $t_i\in S$ such that either $t_iI\subseteq  P_i$ or $t_iB\subseteq B_i$. If $t_iI\subseteq  P_i$ then $$t_i^msB=t_i^ms\left ( \bigcap_{n=1}^{\infty} I^n\right )\subseteq t_i^msI^m\subseteq sP_i^m\subseteq B_i$$
		that is $t_i'B\subseteq B_i$, where $t_i'=t_i^ms$. Combining both cases ( $t_iI\subseteq  P_i$ or $t_iB\subseteq B_i$), we conclude that there exist $u_i\in S$ such that $u_iB\subseteq B_i$ (where $u_i=t_i'$ when $t_iI\subseteq P_i$ and $u_i=t_i$ when $t_iB\subseteq B_i$). Put $u=\prod_{1\leq i\leq k} u_i\in S$, then $uB\subseteq B_i$ for all $1\leq i\leq k$. Consequently, we get $uB\subseteq IB$. By $S$-Nakayama's lemma \cite[Lemma 2.1]{ah2020}, there exist $t\in S$ and $a\in I$ such that $(t+a)B=0$. Hence, $(t+a)b=0$ for all $b\in B$. Since $R$ is $S$-domain and $I\cap S=\emptyset$, therefore $tB=0$. This completes the proof.
\end{proof} 
\begin{cor}
	Let $R$ be an $S$-Noetherian domain and $I$ is an ideal of $R$ disjoint from $S$. If $I\subseteq J(R)$, where $J(R)$ is the Jacobson radical of $R$, then there exist $s\in S$ and $a\in I$ such that $(s+a)\bigcap\limits_{n=1}^{\infty}I^{n}=0$.
\end{cor}  

\begin{proof}
	
	Let $B=\bigcap\limits_{n=1}^{\infty}I^{n}$. Then there exists $u\in S$ such that $uB\subseteq IB$, by Theorem \ref{krull}. Using \cite[Remark 2.1]{ah2020}, $(s+a)\bigcap\limits_{n=1}^{\infty}I^{n}=0$ for some $s\in S$ and $a\in I$.
\end{proof} 
\begin{cor}
	
	If $R$ is an $S$-Noetherian domain, then there exist $s\in S$ and $a\in J_{S}(R)$ such that  $(s+a)\bigcap\limits_{n=1}^{\infty}(J_{S}(R))^{n}=0$, where $J_{S}(R)$ is the $S$-Jacobson radical of $R$.
\end{cor}

	\begin{section}{Conclusion}
		In this paper, we have generalized several results on integral domain and field to $S$-integral domain and $S$-field respectively. We discussed several characterizations of $S$-Noetherian ring with the help of $S$-domains. We also presented the $S$-version of the Krull intersection theorem.
		
		A defining property of a field is that every non-zero element has a unique multiplicative inverse, and vector spaces are traditionally defined over fields. By extending the concept of vector spaces to $S$-vector spaces (i.e., vector spaces over $S$-fields), we can introduce a new framework for analyzing and exploring algebraic structures in commutative algebra. This generalization also opens up new avenues for research in coding theory, cryptography, and other areas where vector spaces and field theory are fundamental. In light of this, we propose the following questions:
		\begin{question}
			How can we define a unique multiplicative inverse of an $S$-non-zero element in a ring with respect to $S$ (say, $S$-inverse)? Furthermore, how can we show that every $S$-non-zero element in an $S$-field possesses an $S$-inverse? This definition must ensure that when the $S$-field is replaced by a conventional field, the $S$-inverse coincides with the usual multiplicative inverse.
		\end{question}
		\begin{question}
			Is it possible to extend the vector space structure over the $S$-field? 
		\end{question}
	\end{section}

    \section*{Acknowledgement}
    The authors express their sincere gratitude to the anonymous reviewers and the editor for their valuable comments and constructive suggestions, which have significantly improved the quality of this work.  G. K. Verma gratefully acknowledges the financial support provided by  UAEU-AUA grant $G00004614$.  Part of this research was carried out while G. K. Verma was affiliated with the Department of Mathematics, Indian Institute of Technology Delhi.


\begin{thebibliography}{999}
		\bibitem{ah18}
		A. Hamed, \textit{On $S$-Mori domains}, J. Algebra Appl., \textbf{17}(2018), 1850171-11. \href{https://doi.org/10.1142/S0219498818501712}{DOI:10.1142/S0219498818501712}.
		
		\bibitem{ah2020} A. Hamed, \textit{A Note on S-Nakayama's Lemma}, Ukr. Math. J., \textbf{72}(2020), 157-160. \href{https://doi.org/10.1007/s11253-020-01769-y}{DOI: 10.1007/s11253-020-01769-y}. 
		
		\bibitem{ah20} 
		H. Ahmed and M. Achraf, \textit{$S$-prime ideals of a commutative ring}, Beitr Algebra Geom., \textbf{61}(2020), 533-542. \href{https://doi.org/10.1007/s13366-019-00476-5}{DOI:10.1007/s13366-019-00476-5}.
        
        \bibitem{ah22} 
		H. Ahmed and M. Achraf, \textit{Some Results on $S$-ACCR Pairs}, Commun. Korean Math. Soc., \textbf{37}(2022), 337-345. \href{https://doi.org/10.4134/CKMS.c210112}{DOI:10.4134/CKMS.c210112}
        
		\bibitem{ad02} 
		D. D. Anderson and T. Dumitrescu, \textit{$S$-Noetherian rings},  Commun. Algebra,  \textbf{30}(2002), 4407-4416. \href{https://www.tandfonline.com/doi/full/10.1081/AGB-120013328}{DOI:10.1081/AGB-120013328}.
		
		\bibitem{ad75}
		D. D.  Anderson, \textit{The Krull Intersection Theorem}, Pac. J. Math., \textbf{57}(1975), 11-14. \href{https://msp.org/pjm/1975/57-1/p02.xhtml}{DOI:10.2140/PJM.1975.57.11}
		
		\bibitem{ad09}
		D. D. Anderson and M. Winders, \textit{Idealization of a module}, J. Commut. Algebra, \textbf{1}(2009), 3-56. \href{https://www.jstor.org/stable/26342937?casa_token=rPO841L4g5gAAAAA%3A8JnqgyWPUgYE-oJov7UUWkJVNoianx2Qwgh2JWOYUW4R-XkarP7i-bzhB431LepIukdVqE_evk-sYmQvP2eRgK_o3fuOKS9jTTxiu9LxkNSqs34vH1Yt&seq=1}{DOI:10.1216/JCA-2009-1-1-3}.
		
		\bibitem{au24}
		A. U. Ansari and S. Prakash, \textit{On S-Principal Ideal Domain}, Palestine Journal of Mathematics, (2024), 136-142. \href{https://pjm.ppu.edu/paper/1723-s-principal-ideal-domain}{DOI:pjm.ppu.edu/paper/1723-s-principal-ideal-domain}
		
		
		
		\bibitem{fm69} 
		M. F. Atiyah and I. G. Macdonald, \textit{Introduction to Commutative Algebra}, Addison-Wesley Publishing Company, (1969). \href{https://www.jstor.org/stable/2316241?origin=crossref&seq=1}{DOI:10.2307/2316241}.
		
		
		
		\bibitem{ba24}
		B. A. Ersoy, U. Tekir, and E. Yildiz, \textit{S-versions and S-generalizations of idempotents, pure ideals and Stone type theorems}, Bull. Korean Math. Soc., \textbf{61}(2024), 83-92. \href{ https://doi.org/10.4134/BKMS.b230023}{DOI:10.4134/BKMS.b230023}.
		
		\bibitem{ki24}
		H. Kim, N. Mahdou, and E. Oubouhou, \textit{On The S-Krull Dimension of a commutative ring}, J. Algebra Appl., (2024).  2650045-20. \href{https://doi.org/10.1142/S0219498826500453}{DOI:10.1142/S0219498826500453}.
		
		\bibitem{hk14}
		H. Kim, M. O. Kim, and J. W. Lim, \textit{On $S$-Strong Mori Domains}, J. Algebra, \textbf{416}(2014), 314-322. \href{https://www.sciencedirect.com/science/article/pii/S0021869314003470?via%3Dihub}{DOI:10.1016/J.JALGEBRA.2014.06.015}.
		
        
		\bibitem{jw14}
		J. W. Lim and D. Y. Oh,  \textit{$S$-Noetherian properties on amalgamated algebras along an ideal}, J. Pure and App. Algebra, \textbf{218}(2014), 1075-1080. \href{https://www.sciencedirect.com/science/article/abs/pii/S0022404913002077?via%3Dihub}{DOI:10.1016/J.JPAA.2013.11.003}.
		
		
		
		
		
        
\bibitem{me22}
        E. Massaoud, \textit{S-primary ideals of a commutative ring}, Commun. Algebra, \textbf{50}(2022), 988-997. \href{https://www.tandfonline.com/doi/full/10.1080/00927872.2021.1977939}{DOI:10.1080/00927872.2021.1977939}.
        
		\bibitem{ap20}
		A. Pekin, U. Tekir, and O. Kilic, \textit{$S$-Semiprime Submodules and $S$-Reduced Modules}, Journal of Mathematics, (2020), 1-7. \href{https://onlinelibrary.wiley.com/doi/10.1155/2020/8824787}{DOI:10.1155/2020/8824787}.
		
		\bibitem{es20}
		E. S. Sevim, U. Tekir, and S. Koc, \textit{S-Artinian rings and finitely S-cogenerated rings}, J. Algebra Appl.,  \textbf{19}(2020), 2050051-2050067. \href{https://www.worldscientific.com/doi/abs/10.1142/S0219498820500516}{DOI:10.1142/S0219498820500516}
		
		\bibitem{es19}
		E. S. Sevim, T. Arabaci, U. Tekir, and S. Koc, \textit{On S-prime submodules}. Turk. J. Math., \textbf{43}(2019), 1036-1046. \href{https://doi.org/10.3906/mat-1808-50}{DOI:10.3906/MAT-1808-50}	
        
		\bibitem{ut25}
        U. Tekir, E. Yildiz, H.A. Khashan, and E. Yetkin Celikel,  \textit{$S$-radical of an ideal and strongly $S$-n ideals}, Ricerche di Mathematics, (2025),  1-17. \href{https://link.springer.com/article/10.1007/s11587-025-00981-x}{DOI:10.1007/s11587-025-00981-x}.
        
		\bibitem{ts23}
		T. Singh, A. U. Ansari, and S. D. Kumar, \textit{$S$-Noetherian Rings, Modules and their generalizations}, Surv. Math. Appl.,  \textbf{18}(2023), 163-182. \href{https://www.utgjiu.ro/math/sma/v18/a18_13.html}{DOI:www.utgjiu.ro/math/sma/v18/ a18-13.html}.
		
		\bibitem{st23}
		T. Singh, A. U. Ansari, and S. D. Kumar, \textit{A Study of $S$-Primary Decompositions}, (2024).	\href{https://arxiv.org/abs/2401.00922}{arXiv:2401.00922}.
		
		\bibitem{ts24}
		T. Singh, A. U. Ansari, and S. D. Kumar, \textit{Existence and uniqueness of $S$-primary decomposition in $S$-Noetherian modules}, Commun. Algebra,  \textbf{52}(2024), 4515-4524. \href{https://www.tandfonline.com/doi/full/10.1080/00927872.2024.2350598}{DOI:10.1080/00927872.2024.2350598}
		
		
		
		
		
		\bibitem{ed20}
		E. Yildiz, B. A. Ersoy, U. Tekir, and S. Koc, \textit{ On S-Zariski topology}, Commun. Algebra \textbf{49}(2021), 1212-1224. \href{https://www.tandfonline.com/doi/full/10.1080/00927872.2020.1831006}{DOI:10.1080/00927872.2020.1831006}

        \bibitem{ed22}
		E. Yildiz, B. A. Ersoy, and  U. Tekir, \textit{On S-Zariski topology on $S$-spectrum of modules}, Filomat, \textbf{36}(2022), 7103-7112. \href{https://doiserbia.nb.rs/Article.aspx?ID=0354-51802220103Y}{DOI:10.2298/fil2220103y}
		
	\end{thebibliography}
\end{document}